\documentclass[11pt]{article}
\usepackage{anysize}\marginsize{3.5cm}{3.5cm}{1.3cm}{2cm}
\makeatletter\@ifundefined{pdfpagewidth}{}{\pdfpagewidth=21.0cm\pdfpageheight=29.7cm}\makeatother 
\usepackage{amssymb,amsmath,graphicx}

\makeatletter
\let\orig@item=\@item \def\@item[#1]{\orig@item[\rm #1]}
\renewenvironment{abstract}{\begin{quote}\footnotesize\textbf{\abstractname.}}{\end{quote}\bigskip}
\renewcommand\@seccntformat[1]{\csname the#1\endcsname.\enspace}
\renewcommand\paragraph{\@startsection{paragraph}{4}{\z@}{1\baselineskip}{-0.5em}{\normalsize\bfseries}}
\newcommand\secparagraph{\@startsection{subsection}{2}{\z@}{1\baselineskip}{-0.5em}{\normalsize\bfseries}}
\let\origcaption=\caption \renewcommand\caption[1]{\parbox{0.66\textwidth}{\origcaption{#1}}}

\renewcommand\@begintheorem[2]{\trivlist\item[\hskip\labelsep{\bfseries#1 #2.}]\it}
\renewcommand\@opargbegintheorem[3]{\trivlist\item[\hskip\labelsep{\bfseries#1 #2}] {\bfseries(#3).}\enspace\it\ignorespaces}
\makeatother

\newtheorem{satz}{Satz}[section]
\makeatletter\@addtoreset{equation}{satz}\makeatother

\newtheorem{proposition}[satz]{Proposition}
\newtheorem{remark}[satz]{Remark}
\newtheorem{algorithm}[satz]{Algorithm}

\newtheorem{introtheorem}{Theorem}

\newenvironment{proof}[1][Proof]{\trivlist\item[\hskip\labelsep{\it #1.}]}{\hspace*{\fill}$\Box$\endtrivlist}

\newcommand\compact{\itemsep=0cm \parskip=0cm}

\newcommand\subjclass[1]{{\renewcommand\thefootnote{}\footnotetext{2010 \textit{Mathematics Subject Classification:} #1.}}}
\newcommand\keywords[1]{{\renewcommand\thefootnote{}\footnotetext{\textit{Keywords.} #1.}}}
\newcommand\engqq[1]{``#1''}
\newcommand\tab[2][t]{\begin{tabular}[#1]{@{}l@{}}#2\end{tabular}}

\renewcommand\emptyset{\varnothing}  
\renewcommand\ge{\geqslant}  
\renewcommand\le{\leqslant}  
\renewcommand\leq{\leqslant}  
\renewcommand\epsilon{\varepsilon}
\renewcommand\phi{\varphi}

\renewcommand\P{\mathbb P}

\newcommand\be{\begingroup\arraycolsep=0.13888em\begin{eqnarray*}}
\newcommand\ee{\end{eqnarray*}\endgroup}
\newcommand\liste[3]{\mbox{$#1_{#2},\dots,#1_{#3}$}}
\newcommand\set[1]{\left\{#1\right\}}
\newcommand\sset[1]{\left\{\,#1\,\right\}}
\newcommand\eqdef{=_{\rm def}}
\newcommand\tensor{\otimes}
\newcommand\union{\cup}
\newcommand\maxmatrcols{10}
\newlength\matrcolsep \matrcolsep=\arraycolsep
\newcommand\matr[1]{{\arraycolsep=\matrcolsep\left(\begin{array}{*{\maxmatrcols}{c}}#1\end{array}\right)}}
\newcommand\tline{\noalign{\vskip0.4ex}\hline\noalign{\vskip0.65ex}}

\newcommand\N{\mathbb N}
\newcommand\Q{\mathbb Q}
\newcommand\R{\mathbb R}
\newcommand\Z{\mathbb Z}
\newcommand\C{\mathbb C}

\newcommand\newop[2]{\newcommand#1{\mathop{\rm #2}\nolimits}}
\newop\NS{NS} 
\newop\vol{vol}
\newop\BigCone{Big} 

\newenvironment{lines}
   {\newcommand\+{\hspace*{1.5em}\ignorespaces} 
   \begin{list}{}{\leftmargin=\leftmargini\parsep=0cm}\item\begin{obeylines}}
   {\end{obeylines}\end{list}}

\begin{document}

\title{Zariski chambers on surfaces of high Picard number}
\author{Thomas Bauer and David Schmitz}
\date{March 13, 2012}
\maketitle
\thispagestyle{empty}

\begin{abstract}
   We present an improved algorithm for the computation of Zariski
   chambers on algebraic surfaces. The new algorithm significantly
   outperforms the so far available method and allows therefore to
   treat surfaces of high Picard number, where huge chamber numbers
   occur. As an application, we efficiently compute the number of
   chambers supported by the lines on the Segre-Schur quartic.
\end{abstract}


\section*{Introduction}

   Zariski chambers are natural pieces into which the big cone of
   an algebraic surface decomposes. Their properties were first
   studied in~\cite{BKS}. It is an intriguing problem, raised
   in~\cite{BFN}, to determine explicitly
   how many Zariski chambers a given surface has. In other words,
   we ask on a smooth projective surface $X$ for the quantity
   $$
      z(X)\eqdef\#\set{\mbox{Zariski chambers on $X$}}\in\N\union\set\infty \,.
   $$
   Roughly speaking, the number $z(X)$ measures how complicated
   the surface $X$ is from the point of view of linear series.
   Specifically, it answers the following natural questions
   (see~\cite{BFN}):
   \begin{itemize}\compact
   \item
      How many different stable base loci can occur in big linear
      series on $X$?
   \item
      How many essentially different Zariski decompositions can
      big divisors on $X$ have? (Here we consider Zariski
      decompositions to be {essentially different}, if their
      negative parts have different support.)
   \item
      How many \engqq{pieces} does the (piecewise polynomial)
      volume function on the big cone of $X$ have?
   \end{itemize}
   To get a more detailed picture of the geometry of $X$, it
   is also natural to consider for integers $d\ge 1$ the
   numbers
   $$
      z_d(X)\eqdef\#\set{{\tab[c]{Zariski chambers on $X$ that
      are\\ supported by curves of degree $\le d$}}} \,,
   $$
   where the degree of curves is taken with respect to a fixed
   ample line bundle on $X$. One has of course $z(X)=\sup_d
   z_d(X)$. While there are surfaces $X$ where $z(X)$ is
   infinite, the numbers $z_d(X)$ are always finite, because on
   any smooth surface there are only finitely many irreducible
   negative curves of fixed degree.

   In~\cite{BFN} an algorithm was presented that computes Zariski
   chambers from the intersection matrix of a set of negative
   curves, and the algorithm was applied to Del Pezzo surfaces. While this
   method is very efficient in these cases, further experience
   has shown that there do exist surfaces where the algorithm
   takes an inordinate amount of time~-- to the point of becoming
   impractical in such situations. This is for instance
   the case when the method is being applied to the 64 lines on
   the Segre-Schur quartic (see~Sect.~\ref{sect-Segre}). At first
   this phenomenon is somewhat surprising, as there are surfaces
   with many more curves to be considered (like the Del Pezzo
   surface $X_8$ with its 240 negative curves) where an
   application of the algorithm poses no practical problems at
   all. It seems that it is not so much the number of
   negative curves that matters most, but the Picard number of
   the surface and the number of chambers that are found. This is
   in accordance with the observation that the essential work
   done by the algorithm (and in fact its essential bottle-neck)
   is the computation of an enormous number of determinants --
   and the dimension of the matrices in question is bounded in
   terms of the Picard number of the surface. (The
   dimension of the determinants to be computed is bounded by
   $\rho(X)-1$.)

   In the present paper we attack this problem by providing a
   significantly improved algorithm that is suited also for
   surfaces with higher Picard number. Our focus here is on the
   efficient calculation of the relevant determinants. As is
   well-known, usual fraction-free algorithms for the computation
   of the determinant of an $n\times n$ integer matrix, like the
   fraction-free Bareiss algorithm, have complexity $O(n^3)$.
   And, to our knowledge, even the currently best fraction-free
   algorithms for determinants over integral domains need
   $O(n^{2.697263})$ ring operations (see~\cite{KalVil}).
   Our
   main point is that within our new algorithm each of the
   necessary determinant computations has a complexity of only
   $O(n^2)$ -- this is achieved by reusing information from
   previous computations (see the details in
   Sect.~\ref{sect-algo}).

   As an application, we determine the
   number of chambers that are supported by the lines
   on the Segre-Schur quartic.
   This remarkable surface
   provides an ideal application for the new
   algorithm: It turns out that it has an enormous
   number of Zariski chambers supported by lines, and it seems
   that the surface lies at the edge of what can be practically
   computed with current methods.

   We show:

\begin{introtheorem}
   Let $X\subset\P^3$ be the Segre-Schur quartic, i.e., the
   surface given in
   homogeneous coordinates $(x_0:x_1:x_2:x_3)$ by the equation
   $$
      x_0(x_0^3-x_1^3)-x_2(x_2^3-x_3^3)=0 \,.
   $$

   {\rm(i)}
   We have
   $$
      z(X)=\infty
   $$
   and
   $$
      z_1(X)=8\,260\,383\,569 \,.
   $$

   {\rm(ii)}
   The maximal number of lines that occur in the support of a
   Zariski chamber is 19
   (which is the maximal possible value, as the Picard number
   of $X$ is 20).
\end{introtheorem}

   Note that the number $z_1(X)$ is bigger by a factor of about
   $10^5$ than the number $1\,501\,681$ of chambers on the Del
   Pezzo surface $X_8$ (blow-up of the plane in 8 points)~-- even
   though only 64 curves are used to build chambers, as opposed to
   240 curves on $X_8$.
   On the other hand, the
   Picard number of the Segre quartic is about twice that of $X_8$.
   One is lead to wonder in general how the number of
   Zariski chambers is related to the Picard number~-- apart from
   the fact that with higher Picard number, chambers of bigger
   support size become theoretically possible.

   Note that since the negative curves on $X$ of higher degrees
   are not known, the numbers $z_d(X)$ are at present
   inaccessible for $d\ge 2$. It would already be very
   interesting to know at what rate they grow for $d\to\infty$.

\bigskip\noindent
   Concerning the organization of this paper, we start in Sect.~1 by
   presenting the improved algorithm. In Sect.~2 we give the
   proof of the theorem on the Segre-Schur quartic
   stated above. Finally, Sect.~3 compares
   the new algorithm with the original one, providing sample
   run-times for the case of Del Pezzo surfaces, the Segre-Schur
   quartic, and for matrices related to Fermat surfaces.


\section{Efficient computation of Zariski chambers}\label{sect-algo}

   Zariski chambers were first studied in~\cite{BKS}, and we
   refer to that paper for a detailed introduction.
   For a very brief account,
   consider a smooth projective surface $X$ over the complex
   numbers. To any big and nef $\R$-divisor $P$ on $X$,
   one associates
   the \textit{Zariski chamber} $\Sigma_P$, which by definition
   consists of all divisor classes $[D]$ in the big cone
   $\BigCone(X)$ such that the
   irreducible curves in the negative part of the Zariski
   decomposition of $D$ are precisely the curves $C$ with $P\cdot
   C=0$. It is the main result of~\cite{BKS} that the sets
   $\Sigma_P$ yield a locally finite
   decomposition of $\BigCone(X)$ into
   locally polyhedral subcones such that
   \begin{itemize}\compact
   \item
      on each subcone the volume function is given by a single
      polynomial of degree two, and
   \item
      in the interior of each of the subcones the stable base
      loci are constant.
   \end{itemize}

   For the purposes of the present paper it is a crucial fact
   that the number of Zariski chambers can be computed from the
   intersection matrix of the negative curves on~$X$, because
   Zariski chambers correspond to negative definite reduced
   divisors~-- one has by~\cite[Prop.~1.1]{BFN}:

\begin{proposition}\label{prop-neg-def}
   The set of Zariski chambers on a smooth projective surface $X$
   that are different from the nef chamber is in bijective
   correspondence with the set of reduced divisors on $X$ whose
   intersection matrix is negative definite.
\end{proposition}
   In the statement, the \textit{nef chamber} is the chamber
   $\Sigma_H$ associated with an ample divisor $H$.
   Its closure is the nef cone, and its interior is the ample
   cone.
   Suppose next that $X$ carries only finitely many negative
   curves (i.e., irreducible curves $C\subset X$ with negative
   self-intersection $C^2$). Then one has, as an immediate
   consequence of Prop.~\ref{prop-neg-def}
   (see~\cite[Prop.~1.5]{BFN}):

\begin{proposition}\label{prop-number-of-chambers}
   \begin{itemize}\compact
   \item[(i)]
   The number $z(X)$ of Zariski chambers on $X$ is given by
   $$
      z(X)=
      1+\#\sset{\tab[c]{
      negative definite principal submatrices of the \\
      intersection matrix of the negative curves on $X$}}\,.
   $$
   \item[(ii)]
   More generally, let $\liste C1r$ be distinct negative curves
   on $X$, and let $S$ be their intersection matrix. Then the
   number of Zariski chambers that are supported by a non-empty
   subset of $\set{\liste C1r}$ equals the number of negative
   definite principal submatrices of the matrix $S$.
   \end{itemize}
\end{proposition}
   Here a \textit{principal submatrix} of a given
   $(n\times n)$-matrix is as usual a submatrix that arises
   by deleting $k$ corresponding rows and columns of the matrix,
   where $0\le k<n$.

   The subsequent algorithm computes the number of positive
   definite principal submatrices of a given symmetric matrix.
   In view of Proposition~\ref{prop-number-of-chambers}, this
   enables us to determine the number of Zariski chambers (by
   considering the negative of the intersection matrix).

\begin{algorithm}\label{algo}\rm
   The algorithm takes as input an integer $n\ge 1$ and a
   symmetric $(n\times n)$-matrix $A$ over $\Z$. It outputs all
   subsets $S\subset\set{1,\dots,n}$ having the property that the
   corresponding principal submatrix $A_S$ is positive definite.

   We adopt a backtracking strategy as in~\cite{BFN}, but instead
   of testing for positive definiteness by a computation of the
   determinant $\det(A_S)$ from scratch, the algorithm makes use
   of three procedures \emph{Grow}, \emph{Shrink}, and
   \emph{IsPosDef} to be discussed below.

   \begin{lines}
      Algorithm \emph{PosDef}
      Input: integer $n\ge 1$, symmetric matrix $A\in\Z^{n\times n}$
      Output: all positive definite principal submatrices of $A$

   \bigskip
      $k \gets 1$; $S\gets\emptyset$; $B\gets()$; $T\gets()$
      \emph{Grow}($S, k$)
      while $S\ne\emptyset$ do
      \+   Assert(($B=T\cdot A_S$) and ($B$ is in Bareiss form))
      \+   if \emph{IsPosDef}($S$) then
      \+\+      output $S$
      \+   else
      \+\+      \emph{Shrink}($S$)
      \+   end if
      \+   if $k < n$ then
      \+\+      $k \gets k + 1$; \emph{Grow}($S, k$)
      \+   else
      \+\+      if $S\ne\emptyset$ and $k=\max S$ then
      \+\+\+         \emph{Shrink}($S$)
      \+\+      end if
      \+\+      if $S\ne\emptyset$ then
      \+\+\+       $k \gets \max S$; \emph{Shrink}($S$); $k \gets k + 1$; \emph{Grow}($S, k$)
      \+\+   end if
      \+   end if
      end while
   \end{lines}

   We now explain the procedures \emph{Grow}, \emph{Shrink}, and
   \emph{IsPosDef}. They work on matrix variables $B$ and $T$
   that are global variables of \emph{PosDef}. At every stage of
   the algorithm, $B$ holds the Bareiss trigonalization of $A_S$.
   The auxiliary matrix variable $T$ is the essential tool that
   makes is possible to do the trigonalization incrementally. It
   holds the triangular matrix that one obtains from the unity
   matrix, when the same transformations are applied that have
   been applied $A_S$ to obtain $B$.

   The procedure \emph{Shrink(S)} removes the maximal element
   from $S$ and updates $B$ and $T$. The latter is achieved by
   simply discarding the last row and column of both $B$ and $T$.

   The procedure \emph{IsPosDef(S)} determines whether $A_S$ is positive
   definite. This can be done as follows: As the matrix
   $A_{S\setminus\max S}$ is known to be positive definite, $A_S$
   is positive definite if and only if its determinant is
   positive. And the latter can be read off the sign of the lower
   right entry of $B$, since in a Bareiss trigonalization this
   entry is always the determinant of the original matrix $A_S$.
   So we have:
   \begin{center}
      \emph{IsPosDef}$(S)= ( B[\max S, \max S] > 0 )$\,.
   \end{center}

   The procedure \emph{Grow(S, k)} includes a new element $k>\max
   S$ to the index set $S$ and updates $B$ and $T$ accordingly.
   It consists of the following steps:

   \begin{enumerate}
   \renewcommand\labelenumi{(G\arabic{enumi})}
   \item
      $S\gets S\union\set k$
   \item
      Attach the last row and the last column of $A_S$ to the bottom and right of $B$.
      Attach the last row and column of the unit matrix to $T$.
   \item
      Replace the last column $b$ of $B$ by $T\cdot b$ .
   \item
      Clear the last row of $B$ by means of the following procedure:
      \begin{lines}
         $s\gets |S|$
         for $i$ from $1$ to $s-1$ do
         \+   if $i = 1$ then
         \+\+      $d \gets 1$
         \+   else
         \+\+      $d \gets B[i-1, i-1]$
         \+   end if
         \+   for $j$ from $i+1$ to $s$ do
         \+\+      $B[s, j] \gets (B[s, j]\cdot B[i, i] - B[i, j]\cdot B[s, i])$ div $d$
         \+   end for
         \+   for $j$ from $1$ to $s$ do
         \+\+      $T[s, j] \gets (T[s, j]\cdot B[i, i] - T[i, j]\cdot B[s, i])$ div $d$
         \+   end for
         \+   $B[s, i] \gets 0$
         end for
      \end{lines}
   \end{enumerate}
   In the two loops of (G4), the division by $d$ is the Bareiss
   divison, made possible by Sylvester's identity, which keeps
   the size of the matrix entries from exploding (cf.~\cite{Bar68}).
   Note that \emph{Grow} does $O(s^2)$ loop iterations and
   hence has complexity $O(n^2)$, while $\emph{Shrink}$ and
   $\emph{IsPosDef}$ need only constant time.
\end{algorithm}


\section{Lines and chambers on the Segre-Schur quartic}\label{sect-Segre}

   In this section we consider the Segre-Schur quartic (see
   \cite{Sch1882}, \cite{Seg44}, and also
   \cite[Sect.~2.1]{Bar85}). By Segre's
   theorem~\cite{Seg43}, this remarkable surface carries the
   maximal number of lines that is possible for a smooth quartic
   in $\P^3$.
   In order to find the chambers supported on lines on this particular
   surface we need to determine the intersection matrix of all lines.
   We approach this task from a more general setting.
   Caporaso, Harris, Mazur \cite{CHM} and Boissiere,
   Sarti \cite{BS07} consider a class of surfaces in $\P^3$ known
   to contain many lines,  which was first studied by Segre \cite{Seg44}.
   Namely, these surfaces are given by an equation
   \begin{equation}\label{eq type}
      \phi(x_0,x_1)=\psi(x_2,x_3),
   \end{equation}
   with homogeneous polynomials $\phi$ and $\psi$ of common
   degree $d$. Clearly the surface we are interested in, the
   Segre-Schur quartic, is of this type with
   $$
   \phi(x,y)=\psi(x,y)=x(x^3-y^3).
   $$
   For any surface $S$ given by an equation (\ref{eq type}) of
   degree $d$ consider the sets of zeros $V(\phi)$ and
   $V(\psi)$ in $\P^1$. Denote by $\Gamma$ the group of
   automorphisms of $\P^1$ mapping $V(\phi)$ onto $V(\psi)$.

   \begin{proposition}[\cite{CHM} Lemma 5.1,
         \cite{BS07} Proposition 4.1]\label{phi-psi}
   The number of lines on $S$ equals
   $d^2+d\cdot\left|\Gamma\right|$.
\end{proposition}

   In order to establish terminology for our further investigation,
   we briefly go through the steps of the proof given in \cite{CHM}.
   \begin{itemize}
   \item
   Consider the sets of points $P_1,\ldots,P_d$ and
   $P_1',\ldots,P_d'$ of $S$ lying on the lines
   $L_1=V(x_2,x_3)$ and $L_2=V(x_0,x_1)$ respectively.
   For every $i,j\in \set{1,\ldots,d}$ the line $L_{i,j}$
   joining $P_i$ and $P_j'$ is contained in $S$. The
   $d^2$ lines $L_{i,j}$ are called \emph{lines of the first type}.
   \item
   Any further line $L$ on $S$ is called a \emph{line
   of the second type}. Any such line is
   disjoint from $L_1$ and $L_2$, guaranteeing that
   it is given by equations of the form
   \begin{eqnarray*}
      x_2&=&ax_0+bx_1\\
      x_3&=&cx_0+dx_1
   \end{eqnarray*}
   with $ad-bc\neq 0$. The invertible matrix
   $\begin{pmatrix}a&b\\c&d
   \end{pmatrix}$
   induces an automorphism of $\P^1$ mapping $V(\phi)$
   onto $V(\psi)$.
   \item
   Conversely, every automorphism in $\Gamma$ given
   by an invertible matrix $  \begin{pmatrix}a&b\\c&d
                  \end{pmatrix}$
   corresponds to $d$ distinct lines
   of the second type on $S$ given by equations
   \begin{eqnarray}
      x_2&=&\lambda\eta^k(ax_0+bx_1)\nonumber\\
      x_3&=&\lambda\eta^k(cx_0+dx_1), \label{lines}
   \end{eqnarray}
   for some $\lambda\in\C$ and $1\le k\le d$
   where $\eta$ denotes
         a primitive $d$-th root of unity.
   \end{itemize}
   With this knowledge we turn to explicitly determining the lines
   lying on the Segre-Schur quartic. The 16 lines of the
   first type are just the lines joining the zeros of the
   polynomial $\phi(x,y)=x(x^3-y^3)$ on $L_1$ and
   $L_2$ each considered as a $\P^1$. Setting
   $\xi=e^{\frac{2\pi i}3}$, these points are
   \begin{align*}
   P_1=(0:1:0:0) \qquad & P_2=(1:1:0:0)\\
   P_3=(\xi:1:0:0) \qquad & P_4=(\xi^2:1:0:0)
  \end{align*}
   and
   \begin{align*}
   P_1^\prime=(0:0:0:1) \qquad & P_2^\prime=(0:0:1:1)\\
   P_3^\prime=(0:0:\xi:1) \qquad & P_4^\prime=(0:0:\xi^2:1).
   \end{align*}
   The lines $L_{i,j}=\overline{P_iP_j^\prime}$
   joining them are given as
   $$
   L_{i,j}:\matr{ -1 & a_i &  0 &  0  \\
                0 & 0   & -1 & a_j }\matr{x_0\\ x_1\\ x_2\\ x_3}=0\,,
   $$
   where $a_i$ denotes the $i$-th entry of the tuple
   $(0,1,\xi,\xi^2)$.

   For the lines of the second type we note that
   in the case of the Segre-Schur quartic $\Gamma$ is
   the tetrahedral group $\mathcal T$, which
   is isomorphic to the product $D_2\times C_3$ of the
   dihedral group $D_2\cong (\Z/2\Z)^2$ and a cyclic
   group $C_3\cong \Z/3\Z$. More concretely, the group
   $D_2$ operating on the points $(0:1),(1:1),(1:\xi),
   (1:\xi^2)$ consists of elements
   $$
   T_1=\matr{1&0\\0&1} \qquad  T_2=\matr{-1&\xi\\2\xi^2&1}
   $$
   $$
   T_3=\matr{-1&\xi^2\\2\xi&1} \qquad  T_4=\matr{-1&1\\2&1} \ ,
   $$
   and the cyclic group $C_3$ is generated by the element
   $$
   Z=\matr{\xi&0\\
          0 &1}.
   $$
   Hence the elements of $\mathcal T$ are the
   automorphisms $Z^{k}T_j$ for $0\leq k\leq 2$
   and \linebreak$1\leq j \leq 4$. For an element
   $\gamma\in\Gamma$ represented by a matrix
   $\begin{pmatrix}a&b\\c&d\end{pmatrix}$
   we need to find the corresponding
   numbers $\lambda_m$ such that the lines
   $$
      L_{\gamma,\lambda_m}: \matr{\lambda_m a&\lambda_m b& -1& 0\\
     \lambda_m c&\lambda_m d& 0 &-1}\matr{x_0\\ x_1\\ x_2\\ x_3}=0\,
   $$
   from equation (\ref{lines}) lie on $S$.
   If $\phi(p_0,p_1)=0$ for a $p\in \mathbb P^1$, then
   $$
     \phi(\lambda a p_0+\lambda bp_1,\lambda cp_0+\lambda
     dp_1)=\lambda^4\phi(ap_0+bp_1,cp_0+dp_1)=0.
   $$
   The line $L_{\gamma,\lambda}$ thus intersects $S$ in
   the four points $(p_0:p_1:\lambda (ap_0+bp_1):\lambda (cp_0+dp_1))$,
   corresponding to the zeros $(p_0:p_1)$ of $\phi$. Now, let
   $(q_0:q_1)\in\mathbb P^1$ be any point on which
   $\phi$ does not vanish. Then the desired values for
   $\lambda$ are given as solutions of the equation
   $$
   \lambda^4=\frac{\phi(q_0,q_1)}{\phi(aq_0+bq_1,cq_0+dq_1)}.
   $$

   By way of example,
   let us carry out the computation for the automorphism
   $Z^1T_2\in\mathcal T$. It is given by the matrix
   $$
   \matr{ -\xi  & \xi^2\\
        2\xi^2 &   1  }.
   $$
   Choosing $(q_0:q_1)=(1:0)$, we get
   $$
    \psi(1,0)=1, \qquad \psi(-\xi,2\xi^2)=-\xi(-\xi^3-8\xi^6)=9\xi.
   $$
   Consequently, the corresponding values for $\lambda_m$ are
   $$
   \lambda_m= \frac{ i^m \xi^2 \sqrt{3}}{3},
   $$
   for $m=0,\ldots ,3$.

   In an analogous manner we find
   the factors
   $\lambda_m$
   for the remaining automorphisms in
   $\mathcal T$
   as shown in Table~\ref{fig:factors}.
   We set $\eta=\sqrt{3}/3$ and write
   $i=\sqrt{-1}$ for the imaginary unit.
   \begin{table}[ht]
      \centering
      \newlength\extraarrayskip \extraarrayskip=2ex
      \begin{tabular}{ccc}\tline
      automorphism & $\qquad$matrix$\qquad$ & $\qquad\lambda_m\qquad$
            \\\tline
            $Z^0T_1$ & $\matr{1&0\\0&1}$ & $  i^m $\\[\extraarrayskip]
            $Z^1T_1$ & $\matr{\xi&0\\0&1}$ & $ i^m\xi^2$\\[\extraarrayskip]
            $Z^2T_1$ & $\matr{\xi^2&0\\0&1}$ & $ i^m\xi$\\[\extraarrayskip]
            $Z^0T_2$ & $\matr{-1&\xi\\2\xi^2&1}$ & $ i^m\eta$\\[\extraarrayskip]
            $Z^1T_2$ & $\matr{-\xi&\xi^2\\2\xi^2&1}$ & $ i^m\eta\xi^2$\\[\extraarrayskip]
            $Z^2T_2$ & $\matr{-\xi^2&1\\2\xi^2&1}$ & $ i^m\eta\xi$\\[\extraarrayskip]
            $Z^0T_3$ & $\matr{-1&\xi^2\\2\xi&1}$ & $ i^m\eta$\\[\extraarrayskip]
            $Z^1T_3$ & $\matr{-\xi&1\\2\xi&1}$ & $ i^m\eta\xi^2$\\[\extraarrayskip]
            $Z^2T_3$ & $\matr{-\xi^2&\xi\\2\xi&1}$ & $ i^m\eta\xi$\\[\extraarrayskip]
            $Z^0T_4$ & $\matr{-1&1\\2&1}$ & $ i^m\eta$\\[\extraarrayskip]
            $Z^1T_4$ & $\matr{-\xi&\xi\\2&1}$ & $ i^m\eta\xi^2$\\[\extraarrayskip]
            $Z^2T_4$ & $\matr{-\xi^2&\xi^2\\2&1}$ & $ i^m\eta\xi$\\[\extraarrayskip]
            \tline
      \end{tabular}
      \caption{\label{fig:factors}%
         Factors $\lambda_m$ for the automorphisms in
         $\mathcal T$}
   \end{table}

   The lines $L_{Z^{k}T_j,\lambda_m}$ for
   $k=0,\ldots,2$, $j=1,\ldots,4$, $m=0,\ldots,3$
   are thus exactly the 48 lines of the second type on $S$.

   The configuration of the lines on the Segre-Schur
   quartic is given by the following
   \begin{proposition}
   Let $\delta: \R\to \set{0,1}$ denote the
   indicator function of the set $\set{0}$ mapping any non-zero
   number $x$ to $0$ and the number $0$ to $1$. With
   further notation as above, we have:
   \begin{itemize}
   \item Every line $L$ on $S$ has self-intersection
      $L^2=-2$.
   \item
    The intersection number of distinct lines of the
    first type is
   $$
      L_{i,j}\cdot L_{i',j'} = \delta(i-i')+\delta(j-j').
   $$
   \item The intersection number of a line of the
      first type with a line of the second type is
   $$
   L_{s,t}\cdot L_{Z^{k}T_j,\lambda_m} = \begin{cases}
   \delta(-\xi^{2k+t-1}+\xi^j-\xi^{s-1}-
      2\xi^{s+t+2k-j-2})& \text{, if } j\neq1,s\neq1,t\neq1,\\
      \delta(t+2k-j-1 \mod 3) &  \text{, if } j\neq1,t\neq1,s=1,\\
      \delta(s-1-j \mod 3) & \text{, if } j\neq1,s\neq1,t=1,\\
      \delta(s+k-t \mod 3) & \text{, if } j=1,s\neq1,t\neq1,\\
      \delta(s-1)\cdot\delta(j-1)\cdot\delta(t-1) & \text{, else }
      \end{cases}\\
   $$
   \item The intersection number of distinct lines of the
      second type is
   $$
      L_{Z^{k}T_j,\lambda_m}\cdot L_{Z^{k'}T_{j'},\lambda_{m'}}
      =  \begin{cases}
      \delta((i^{m'-m}-3)\xi^{2k}\\\quad
      +i^{m'-m}(1-3i^{m'-m})\xi^{2k'}\\\quad
      +2i^{m'-m}(\xi^{2k-j+j'}+\xi^{2k'-j'+j})) & \text{, if }
         j\neq1,j'\neq1,\\
      \delta(i^{2m}\xi^{2k}-i^mi^{m'}\xi^{2k'}\eta
      +i^{m'}\eta i^m\xi^{2k}
      \\\quad-3({i^{2m'}}\eta^2\xi^{2k'})) &    \text{, if } j=1,j'\neq1,\\
      \delta(m-m') & \text{, if } j=j'=1.
      \end{cases}\\
   $$
   \end{itemize}
   \end{proposition}

   \begin{proof}
   The statement about self-intersection of
   lines on $S$ follows immediately by
   adjunction and the fact that  $S$ has trivial
   canonical class. To prove the statement
   about the intersection of distinct lines of the first
   type, all we need to show is that two such
   lines $A,B$ cannot intersect in a point
   outside the (disjoint) lines $L_1$ and $L_2$
   (see sketch of proof of Proposition \ref{phi-psi}).
   But if this were the case, since $L_1$ and $L_2$
   both intersect $A$ and $B$, they would lie in the
   plane spanned by $A$ and $B$ and would
   therefore intersect each other.

   The remaining intersection numbers are calculated
         in the following way: Two lines in $\P^3$ given by
   equations
   \begin{eqnarray*}
      a_0x_0+a_1x_1+a_2x_2+a_3x_3&=&0\\
      b_0x_0+b_1x_1+b_2x_2+b_3x_3&=&0
   \end{eqnarray*}
   and
   \begin{eqnarray*}
      c_0x_0+c_1x_1+c_2x_2+c_3x_3&=&0\\
      d_0x_0+d_1x_1+d_2x_2+d_3x_3&=&0
   \end{eqnarray*}
   intersect each other if and only if the determinant
   of the matrix
   $$
   \matr{a_0&a_1&a_2&a_3&\\
           b_0&b_1&b_2&b_3&\\
           c_0&c_1&c_2&c_3&\\
           d_0&d_1&d_2&d_3&}
   $$
   vanishes.
   Note that for $s,t\neq1$ the line $L_{s,t}$ of the
   first type is given by
   $$
   L_{s,t}:\matr{ -1 & \xi^{s-2} &  0 &  0  \\
                0 & 0   & -1 & \xi^{t-2} }\matr{x_0\\ x_1\\ x_2\\ x_3}=0\,,
   $$
   and if $j\neq1$, then the line $L_{Z^{k}T_j,\lambda_m}$
   is given by
   $$
      L_{Z^{k}T_j,\lambda_m}:
      \matr{ -i^m\eta & i^m\eta\xi^{j-1} &  -1 &  0  \\
      2i^m\eta\xi^{2k+1-j} & i^m\eta\xi^{2k}   & 0 & -1 }\matr{x_0\\ x_1\\ x_2\\ x_3}=0\,.
   $$
   The computation of the determinants in question
   then yields the asserted formulas.
   Let us by way of illustration calculate the
   intersection number of the line $L_{1,t}$ of the first
   type, where $t\neq1$, and a line  $L_{Z^{k}T_j,\lambda_m}$
   of second type, with $j\neq1$. We have
   \begin{eqnarray*}
      \det\matr{ -1&0&0&0\\
      0&0&-1&\xi^{t-2}\\
      -i^m\eta&i^m\eta\xi^{j-1}&-1&0\\
      2i^m\eta\xi^{2k+1-j}& i^m\eta\xi^{2k}&0&-1 \\ } &=&
          i^m\eta(\xi^{j-1}-\xi^{2k+t-2}).
   \end{eqnarray*}
   Consequently, the lines intersect if and only if
   $$
      \xi^{j-1}=\xi^{2k+t-2},
   $$
   which is equivalent to
   $$
      t+2k-j-1 \equiv 0 \mod 3.
   $$
   \end{proof}

   Combining all the intersection numbers yields the
   intersection matrix of the Segre-Schur quartic
   displayed in Figure~\ref{fig:Schnittmatrix}.
   The order of rows and columns
   was chosen as follows: The first 16 rows
   correspond to the lines of the first type $L_{s,t}$
   with $s=1,\ldots,4$, $t=1,\ldots,4$. The remaining 48
   rows correspond to the lines of the second type
   $L_{Z^{k}T_j,\lambda_m}$ with $k=0,\ldots,2$,
   $m=0,\ldots,3$, $j=1,\ldots,4$.
   \begin{figure}[t]
      \centering
      \includegraphics[width=0.65\textwidth]{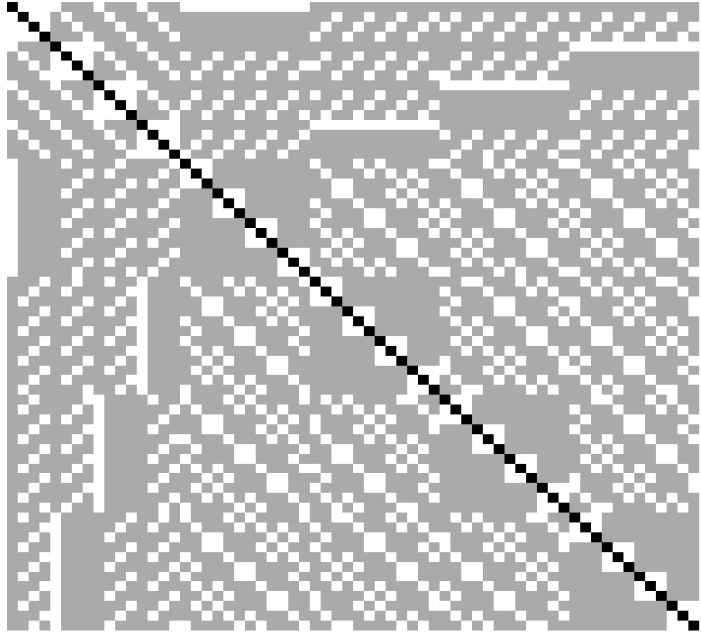}
      \caption{\label{fig:Schnittmatrix}%
         Intersection matrix of the 64 lines on the Segre-Schur
         quartic.
         The numerical entries of
         $-2$, $0$ and $1$ are replaced by black, grey and white boxes
         respectively.}
   \end{figure}

\begin{proof}[Proof of the theorem stated in the introduction]
   The intersection matrix of the 64 lines is of rank 20, hence
   the classes of the lines generate the N\'eron-Severi group
   $\NS(X)$ over $\Q$, and the K3 surface is singular (in the
   sense that $\NS(X)$ is of maximal possible rank). Therefore,
   by a result of Shioda and Inose~\cite{ShiIno},
   its automorphism group is infinite. The latter
   fact in turn implies that there are infinitely many
   $(-2)$-curves on $X$ (cf.~\cite[Remark 7.2]{Kov94}), hence
   $z(X)=\infty$.

   An application of Algorithm~\ref{algo} to the negative of the
   intersection matrix of the 64 lines shows that it has exactly
   8\,260\,383\,568
   positive
   definite principal submatrices. By taking the additional nef
   chamber into account we arrive at the claimed
   8\,260\,383\,569
   chambers.
   As $\rho(X)=20$, there cannot be Zariski chambers supported
   by more than 19 curves. The results of the algorithm show
   that chambers with 19 supporting curves actually occur. (In fact there
   are exactly 1728 such chambers, see the subsequent
   remark~\ref{remark-numbers-by-support}.)
\end{proof}

\begin{remark}\label{remark-numbers-by-support}\rm
   There is more numerical data that is of interest: How many
   Zariski chambers are there with support of given cardinality?
   Let us denote for $\ell\ge 1$ by $z_d^{(\ell)}(X)$ the number
   of Zariski chambers that are supported by $\ell$ curves of
   degree~$\le d$. Our computations yield for the Segre-Schur quartic
   the values in the following table.
   $$
      \begin{array}{rr}\tline
          \ell & z_1^{(\ell)}(X) \\ \tline
          1  & 64         \\
          2  & 2016       \\
          3  & 41376      \\
          4  & 605856     \\
          5  & 6343776    \\
          6  & 45613512   \\
          7  & 217025520  \\
          8  & 674047818  \\
          9  & 1376161536 \\
         10  & 1900843848 \\
         11  & 1832006112 \\
         12  & 1264421472 \\
         13  & 635795760  \\
         14  & 233619648  \\
         15  & 61499712   \\
         16  & 11037702   \\
         17  & 1246368    \\
         18  & 69744      \\
         19  & 1728       \\ \tline
      \end{array}
   $$
   The numbers show that the surface clearly favours
   chambers of \engqq{medium size}.
\end{remark}

\begin{remark}\rm
   The lines
   $$
      L_{1,1},\ldots,L_{1,4},L_{2,1},L_{2,2},L_{2,3},L_{3,1},L_{3,2},L_{3,3}
   $$
   of the first type together with the lines
   \begin{eqnarray*}
      L_{Z^1T_1\lambda_0},L_{Z^2T_1\lambda_0},L_{Z^1T_2\lambda_0},
      L_{Z^2T_2\lambda_0},L_{Z^1T_3\lambda_0},\\L_{Z^2T_3\lambda_0},
      L_{Z^1T_1\lambda_1},L_{Z^2T_1\lambda_1},L_{Z^1T_2\lambda_1},
      L_{Z^2T_2\lambda_1}
   \end{eqnarray*}
   of the second type
   generate a sublattice $\Lambda\subset\NS(X)$ of rank 20 and
   discriminant~$-48$. So they yield a basis of
   $\NS(X)\tensor\Q$.
   It is shown in \cite[Appendix~B]{BS08} that
   the discriminant of the lattice $\NS(X)$ is $-48$ as well,
   hence the mentioned lines
   in fact generate $\NS(X)$.
\end{remark}

\begin{remark}\rm
   It is interesting to note that the 16 lines of the first type
   that one has on
   any quartic surface of type
   $$
      \phi(x_0,x_1)=\psi(x_2,x_3)
   $$
   give rise to 6521 Zariski chambers.
   As these 16 lines generate a sublattice of rank 10, one should
   compare the number 6521 with the number $z(X_8)=1\,501\,681$
   for the Del Pezzo surface $X_8$ of Picard number~9. It would
   be interesting to know on a more conceptual basis what
   properties of the lattice are crucial for obtaining a large
   number of chambers.
\end{remark}


\section{Comparison of efficiency}

   In order to illustrate the practicability of the new
   algorithm, we will now compare it with the algorithm from
   \cite{BFN} by providing sample run-times.\footnote{The
   run-times were obtained using \textsc{Delphi}
   implementations of $A_1$
   and $A_2$ on an Intel Core Duo CPU E8600 at 3.33\,GHz.}

   We consider first the Del Pezzo surfaces $X_1,\dots,X_8$ as
   in~\cite{BFN}. The following table lists their chamber numbers
   $z(X_r)$, and it shows the run-times (all in milliseconds) for
   the algorithm from~\cite{BFN}, called $A_1$ in the table, and
   for the new algorithm $A_2$. The value $n$ is the dimension of
   the intersection matrix\footnote{Here and in the sequel the
   algorithm is actually applied to the \textit{negative} of the
   intersection matrix in question.}
   (i.e., the number of negative curves
   on $X_r$).
   $$
      \begin{array}{c|*8c}\tline
         r & 1 & 2 & 3 & 4 & 5 & 6 & 7 & 8 \\ \tline
         n & 1 & 3 & 6 & 10 & 16 & 27 & 56 & 240 \\
         z(X_r)   & 2 & 5 & 18 & 76 & 393 & 2\,764 & 33\,645 & 1\,501\,681 \\ \tline
         A_1  & 0.55 &0.56&0.7 &0.85&2.3&33.5&1.48\times10^3&2.88 \times 10^4 \\
         A_2 & 0.8  &0.85&0.82&0.91&2.3&26.8&1.09\times10^3&1.92\times10^4 \\ \tline
          \text{Factor} & 0.69 & 0.66 & 0.85& 0.93 & 1 & 1.25 & 1.36& 1.5 \\ \tline
      \end{array}
   $$
   For $r\le 4$ the original algorithm $A_1$ is actually faster,
   presumably due to the overhead caused by the more
   sophisticated strategy of $A_2$. Starting with $r=6$, however,
   $A_2$
   gets more and more superior. This pattern will become even
   more visible when we now consider the
   lines on the Segre quartic. Specifically, we
   provide run-times for $A_1$ and $A_2$, when applied to the
   principal submatrices consisting of the first $n$ rows and
   columns of the intersection matrix from
   Sect.~\ref{sect-Segre}.
   \begin{center}\footnotesize
     $\begin{array}{c|*{20}c}\tline
         n        & 8 & 16 & 24 & 32 & 40 & 48 & 56 & 64 \\\tline
         A_1 &0.9& 65,55&4.16\times10^3&9.46\times10^4&1.43 \times10^6&1.57\times 10^7&1.3\times10^8& *   \\
         A_2 &0.96& 31.54&1.7\times10^3&3.61\times10^4&5.4\times10^5&5.66\times10^6 &4.6\times10^7&4.9\times10^8\\ \tline
         \text{Factor} & 0.94 & 2.08 & 2.45& 2.62 & 2.65 & 2.77 & 2.83 & * \\ \tline
      \end{array}$
   \end{center}
   Clearly, with growing matrix dimension algorithm $A_2$ shows its
   advantage over algorithm $A_1$.
   (As for the asterisk in the last column: The factors in the
   fourth line of the table seem to
   suggest a factor of about 3. If this
   assumption is true, then $A_1$ should for $n=64$
   have a run-time of about
   $15\times 10^8\,\mbox{ms}$, which is around 3
   weeks. In our attempts to verify the assumption, we were
   however not able to get results within that amount of time,
   and we found it technically difficult to check for run-times
   beyond a month.)

   The potential of $A_2$ is demonstrated
   even more clearly in the case of matrices with large definite
   principal submatrices, e.g., matrices which are negative definite
   themselves. We use a symmetric integer matrix with negative entries
   $-2$ on the main diagonal, super- and subdiagonal entries 1, and
   remaining entries 0. Note that a similar matrix can in fact be
   realized as an intersection matrix: For any number $k$ consider a
   surface of the type considered in section \ref{sect-Segre} of
   degree $k$, e.g., the Fermat surface of degree $k$.
   Among the $k^2$ lines of the first type we can pick
   $n=2k+1$ lines with the desired configuration.
   However, the self-intersections of these lines (and therefore the
   diagonal entries of the matrix) are then $2-k$.
   $$
      \begin{array}{c|*{20}c}\tline
         n    &   15 & 21 & 23 & 27 &33 &  \\\tline
         A_1 &    2.73\times10^2 & 3.84\times10^4&4.26\times10^5& 4.5\times10^6 &4.74\times10^8&  \\
         A_2 &    1.1\times10^2 & 1.17\times10^4&1.18\times10^5& 1.23\times10^6&1.06\times10^8&\\ \tline
         \text{Factor} & 2.48 & 3.28&    3.61    & 3.66& 4.48\\ \tline
      \end{array}
   $$



\bigskip
\small
   Tho\-mas Bau\-er,
   Fach\-be\-reich Ma\-the\-ma\-tik und In\-for\-ma\-tik,
   Philipps-Uni\-ver\-si\-t\"at Mar\-burg,
   Hans-Meer\-wein-Stra{\ss}e,
   D-35032~Mar\-burg, Germany.

\nopagebreak
   \textit{E-mail address:} \texttt{tbauer@mathematik.uni-marburg.de}

\bigskip
   David Schmitz,
   Fach\-be\-reich Ma\-the\-ma\-tik und In\-for\-ma\-tik,
   Philipps-Uni\-ver\-si\-t\"at Mar\-burg,
   Hans-Meer\-wein-Stra{\ss}e,
   D-35032~Mar\-burg, Germany.

\nopagebreak
   \textit{E-mail address:} \texttt{schmitzd@mathematik.uni-marburg.de}



\begin{thebibliography}{99}\footnotesize\itemsep=0cm\parskip=0cm

\bibitem{Bar68}
   Bareiss, E. (1968):
   Sylvester's identity and multistep integer-preserving Gaussian elimination.
   Mathematics of computation 22, 565-578 (1968)

\bibitem{Bar85}
   Barth, W.:
   Lectures on K3- and Enriques surfaces.
   Algebraic Geometry (Sitges, 1983), Lect. Notes Math. 1124, Springer-Verlag, 1985, pp. 21-57

\bibitem{BFN}
   Bauer, Th., Funke, M., Neumann, S.:
   Counting Zariski chambers on Del Pezzo surfaces.
   Journal of Algebra 324, 92-101 (2010)

\bibitem{BKS}
   Bauer, Th., K\"uronya, A., Szemberg, T.:
   Zariski chambers, volumes, and stable base loci.
   J. reine angew. Math. 576, 209-233 (2004)

\bibitem{BS07}
   Boissi\`ere, S., Sarti, A.:
   Counting lines on surfaces,
   Ann. Sc. Norm. Super. Pisa, Cl. Sci. 6 , 39--52 (2007)

\bibitem{BS08}
   Boissi\`ere, S., Sarti, A.:
   On the N\'eron-Severi group of surfaces with many lines.
   Proceedings of the American Mathematical Society 136, 3861-3867 (2008).

\bibitem{CHM}
   Caporaso, L.,Harris, J., Mazur, B.:
   How many rational points can a curve have?, The moduli space of curves.
   (Texel Island, 1994), Progr. Math., vol. 129,
   Birkh\"auser, Boston, 13-31 (1995)

\bibitem{KalVil}
   Kaltofen, E., Villard, G.:
   On the complexity of computing determinants.
   comput. complex. 13 (2004), 91--130

\bibitem{Kov94}
   Kov\'acs, S.J.:
   The cone of curves of a K3 surface.
   Math. Ann. 300, 681-691 (1994)

\bibitem{Sch1882}
   Schur, F.:
   Ueber eine besondre Classe von Fl\"achen vierter Ordnung.
   Math. Ann. 20, 254-296 (1882)

\bibitem{Seg43}
   Segre, B.:
   The maximum number of lines lying on a quartic surface.
   Oxf. Quart. J. 14, 86-96 (1943)

\bibitem{Seg44}
   Segre, B.:
   On arithmetical properties of quartic surfaces.
   Proc. London Math. Soc. 49, 353-395 (1944)

\bibitem{ShiIno}
   Shioda, T., Inose, H.:
   On singular K3 surfaces.
   In: Complex Anal. algebr. Geom., Collect. Pap. dedic. K. Kodaira, 119-136 (1977).

\end{thebibliography}
\end{document}